\documentclass[a4j,11pt,draft]{article}
\usepackage{amsmath}
\usepackage{amssymb}
\usepackage{amsthm}
\usepackage[mathscr]{eucal}
\usepackage{tikz}
\usepackage[all,knot,cmtip]{xy}
\usepackage{color}
\xyoption{arc}
\entrymodifiers={+!!<0pt,\fontdimen22\textfont2>} 

\numberwithin{equation}{section}
\newtheorem{definition}{Definition}[section]
\newtheorem{prop}[definition]{Proposition}
\newtheorem{lem}[definition]{Lemma}

\newtheorem{thm}[definition]{Theorem}

\newtheorem{cor}[definition]{Corollary}

\vspace{1mm}

\title{Verschiebung maps among $K$-groups of truncated polynomial algebras}
\author{Ryo Horiuchi}
\date{}
\begin{document}

\newpage
\maketitle


\section{Introduction}
Let $p$ be a prime number, and let $A$ be a ring in which $p$ is nilpotent. In this paper, we consider the maps $$K_{q+1}(A[x]/(x^m), (x))\to K_{q+1}(A[x]/(x^{mn}), (x)),$$induced by the ring homomorphism $A[x]/(x^{m})\to A[x]/(x^{mn})$, $x\mapsto x^n$. We evaluate these maps, up to extension, for general $A$ in terms of topological Hochschild homology, and for regular $\mathbb{F}_p$-algebras $A$, in terms of groups of de Rham-Witt forms. After the evaluation, we give a calculation of the relative $K$-group of $\mathcal{O}_{K}/p\mathcal{O}_{K}$ for certain perfectoid fields $K$.

By Dundas-Goodwillie-McCarthy (\cite[Theorem 7.0.0.2]{DGM}), relative algebraic $K$-theory and relative topological cyclic homology coincide for algebras which we are studying. In \cite[Theorem 4.2.10]{HM1}, Hesselholt and Madsen constructed an exact sequence of algebraic $K$-theory for polynomial algebras and gave an explicit presentation (\cite[Theorem A]{HM1}) of the algebraic $K$-theory of a truncated polynomial algebra over a perfect field of positive characteristic by the Witt vectors. In \cite[Theorem A]{H1}, Hesselholt constructed a tower of such long exact sequences. The tower corresponds to the canonical projections $A[x]/(x^m)\to A[x]/(x^n)$ with $m\geq n$ for $A$ an $\mathbb{F}_p$-algebra. Although their results are expressed in the language of de Rham-Witt complexes and algebraic $K$-theory, what they have used are topological Hochschild homology and topological cyclic homology of ring spectra. Using the language of ring spectra and their homology theories, we consider the map of relative algebraic $K$-groups $$v_n \colon K_{q+1}(A[x]/(x^m), (x))\to K_{q+1}(A[x]/(x^{mn}), (x))$$ induced by the ring homomorphism $$A[x]/(x^{m})\to A[x]/(x^{nm}), \ \ x\mapsto x^n.$$ We also consider the colimit over the $v_n$'s indexed by the filtered category $(\mathbb{N}, *)$ of natural numbers under multiplication. One of our main results is the following (Theorem \ref{THM1}). The TR-groups are explained in Section 4.
\begin{thm}Let $A$ be a ring in which $p$ is nilpotent. There is a map of long exact sequences
\[
   \xymatrix{
   \vdots\ar[d]\ar[d]& \vdots\ar[d]&\\
    \lim_{R} \operatorname{TR}^{i/k}_{q-\lambda_{d_{i,k}}}(A) \ar[r]^-{\operatorname{id}} \ar[d]^{{V_k}_*} & \lim_{R} \operatorname{TR}^{i/kn}_{q-\lambda_{d_{i, kn}}}(A)  \ar[d]^-{{{V}_{kn}}_*}  \\  
    \lim_{R} \operatorname{TR}^{i}_{q-\lambda_{d_{i,k}}}(A) \ar[d]    \ar[r]^-{{V_{n}}_*} & \lim_{R} \operatorname{TR}^{i}_{q-\lambda_{d_{i, kn}}}(A)  \ar[d] \\
   K_{q+1}(A[x]/(x^k), (x))  \ar[d] \ar[r]^{v_n}& K_{q+1}(A[x]/(x^{nk}), (x)) \ar[d] \\
  \vdots&\vdots &,\\
}
\]
where $d_{i,k}=\lfloor(i-1)/k\rfloor$ is the integer part of $(i-1)/k$, $V$'s denote the maps induced by Verschiebung maps, the maps in the limits are restriction maps, and the maps $v_n$ are induced by $x\mapsto x^n$.
\end{thm}

One of the remarkable facts about topological Hochschild homology is that its homotopy groups correspond to the groups of de Rham-Witt forms. Via this correspondence, stable homotopy theory and $p$-adic Hodge theory have been getting closer each other. Considering it, we also produce a diagram similar to the one above for regular algebras over $\mathbb{F}_p$ in terms of de Rham-Witt groups via the translation between topological Hochschild homology and the de Rham-Witt complex given by \cite[$\S$1, $\S$2, $\S$5]{H1}. See also \cite{H2}.

\begin{thm}Let $A$ be a regular $\mathbb{F}_p$-algebra. There is a map of long exact sequences\vspace{-2mm}
\[
   \xymatrix{
   \vdots\ar[d]&\vdots\ar[d]&\\
   \bigoplus_{l\geq0}\mathbb{W}_{l+1}\Omega_{A}^{q-2l} \ar[d]^-{{V_k}_*} \ar[r]^-{\operatorname{id}}       &       \bigoplus_{l\geq0}\mathbb{W}_{l+1}\Omega_{A}^{q-2l}      \ar[d]^-{{V_{kn}}_*}  \\  
     \bigoplus_{l\geq0}\mathbb{W}_{k(l+1)}\Omega_{A}^{q-2l}                      \ar[r]^-{{V_n}_*} \ar[d]&           \bigoplus_{l\geq0}\mathbb{W}_{kn(l+1)}\Omega_{A}^{q-2l}   \ar[d]\\
   K_{q+1}(A[x]/(x^k), (x)) \ar[r]^-{v_n} \ar[d] & K_{q+1}(A[x]/(x^{nk}), (x)) \ar[d] \\
 \vdots& \vdots &,\\
}
\]
where the subscript $m(l+1)$ denotes the truncation set $\{1, 2, ..., m(l+1)\}$ and $\mathbb{W}_{(-)}\Omega_A^*$ denotes the big de Rham-Witt complex over $A$.
\end{thm}

As a consequence of the above theorems, we calculate some relative $K$-groups. Let $k$ be a perfect field of characteristic $p$. We define $$\mathcal{O}_K:=(\operatorname{colim}_nW(k)[p^{1/p^n}])^{\wedge}$$ with $\mathfrak{m}_K\subset\mathcal{O}_K$ the maximal ideal, and $$\mathcal{O}_L:=(\operatorname{colim}_nW(k)[\zeta_{p^n}])^{\wedge}$$ with $\mathfrak{m}_L\subset\mathcal{O}_L$ the maximal ideal. With these notations, we have
\[
   \xymatrix{
   K_{2j-1}(\mathcal{O}_K/p\mathcal{O}_K, \mathfrak{m}_K/p\mathcal{O}_K)=\operatorname{colim}_n(\mathbb{W}_{p^nj}(k)/V_{p^n}\mathbb{W}_j(k)),\\
   }
\]
   \[
   \xymatrix{
   K_{2j-1}(\mathcal{O}_L/p\mathcal{O}_L, \mathfrak{m}_L/p\mathcal{O}_L)=\operatorname{colim}_n(\mathbb{W}_{p^{n-1}(p-1)j}(k)/V_{p^{n-1}(p-1)}\mathbb{W}_j(k)),\\      
}
\]
where the colimits are indexed by the category of natural numbers under addition and $\mathbb{W}$ denotes the big Witt vectors. Moreover, the relative $K$-groups in even degrees are zero.

\section{The cyclic bar-construction}
We recall some notations from \cite{H1} or \cite{HM1}. For a positive integer $k$, we let $\Pi_k:=\{0, 1, x^1, ..., x^{k-1}\}$ denote the finite commutative pointed monoid defined by $x^nx^m=x^{m+n}$, $0x^n=0$, $x^0=1$ and $x^k=0$. We let $\mathbb{T}$ denote the circle group and $\lambda_d:=\mathbb{C}(d)\oplus...\oplus\mathbb{C}(1),$ where $\mathbb{C}(j)=\mathbb{C}$ as linear spaces with the $\mathbb{T}$-action defined by $\mathbb{T}\times\mathbb{C}(j)\to\mathbb{C}(j), (z, w)\mapsto z^jw,$ for $1\leq j\leq d$. We let $\operatorname{N^{cy}}(\Pi_k)$ denote the geometric realization of the cyclic set $\operatorname{N^{cy}}(\Pi_k)[-]$ defined by cyclic bar construction of $\Pi_k$. This pointed space has the canonical decomposition 

$$\operatorname{N^{cy}}(\Pi_k)\cong\bigvee_{i\geq0}\operatorname{N^{cy}}(\Pi_k, i),$$ induced from the canonical decomposition of the cyclic set. For example, $\operatorname{N^{cy}}(\Pi_k, 0)$ is the 2-point space $\{0,1\}$. In \cite[$\S$3]{HM1}, the homotopy classes of the following maps of pointed $C_i$-spaces are defined; $\theta_d:\Delta^{i-1}/C_i\cdot\Delta^{i-k}\to S^{\lambda_d}$ for $kd<i<k(d+1)$, and $\theta_d:\Delta^{i-1}/C_i\cdot\Delta^{i-k}\to (S^0\ast C_k)\wedge S^{\lambda_d}$ for $i=k(d+1)$, where $C_m$ is the $m$-th cyclic group and $S^{\lambda_d}$ is the one point compactification of $\lambda_d$. They play a key role in this paper. In op. cit., for any positive integer $i$, the following cofibration sequence are constructed using $\theta$;
\[
   \xymatrix{
  \mathbb{T}_+\wedge_{C_{i/k}}S^{\lambda_{d_i}} \ar[r]^-{\operatorname{pr}} & \mathbb{T}_+\wedge_{C_i}S^{\lambda_{d_i}} \ar[r] & \operatorname{N^{cy}}(\Pi_k, i) \ar[r] & \Sigma\mathbb{T}_+\wedge_{C_{i/k}}S^{\lambda_{d_i}}, 
  }
\]
where $\mathbb{T}_+\wedge_{C_{i/k}}S^{\lambda_{d_i}}$ is trivial when $k$ does not divide $i$ and $d_i=\lfloor(i-1)/k\rfloor$ is the largest natural number less than or equal to $(i-1)/k$. We briefly recall the construction.


For the $i$-th cyclic group $C_i$, $\mathbb{R}[C_i]$ denotes the regular representation and $\Delta^{i-1}\subset\mathbb{R}[C_i]$ denotes the convex hull of the generators of $C_i$. By permutation $C_i$ acts on $\mathbb{R}[C_i]$ and the action restricts on $\Delta^{i-1}$. Let $\xi_i$ denote the generator of $C_i$ and $\Delta^{i-m}$ the convex hull of $1, \xi_i, ..., \xi_i^{i-m}$. The canonical decomposition of $\mathbb{R}[C_i]$ induces the projection map (\cite[p.11]{H1})
$$\pi_d:\mathbb{R}[C_i]\to \lambda_d,$$ if $2d<i$. We first consider the case $md<i<m(d+1)$. In \cite[$\S$3]{HM1}, it is proved that $0\not\in\pi_d(C_i\cdot\Delta^{i-m})\subset\lambda_d$. Composing $\pi_d|_{\Delta^{i-1}}$ and the radial projection, we get a $C_i$-equivariant map
$$\theta_d:\Delta^{i-1}/C_i\cdot\Delta^{i-m}\to S^{\lambda_d}.$$ We next consider the case $i=m(d+1)$. It is also proved that in \cite[$\S$3]{HM1} $0\not\in\pi_{d+1}(C_i\cdot\Delta^{i-m})\subset\lambda_{d+1}$. Furthermore, that proves $$\pi_{d+1}(C_i\cdot\Delta^{i-m})\cap\lambda_d^\perp=C_m',$$ where $\lambda_d^\perp$ is the orthogonal completion of the image of the canonical inclusion $\lambda_d\xrightarrow\iota\lambda_{d+1}$ and $C_m'$ is the preimage of $C_m$ by the isomorphism $\lambda_d^\perp\to \mathbb{C}(d+1)$ induced by $\iota$. Picking a small ball $B\subset\lambda_{d+1}\backslash C_m'$ around a point in the sphere $S(\lambda_d^\perp)$, we define $U:=(C_i\cdot B)\cap S(\lambda_{d+1})$. If $B$ is small enough, the projection $\pi_{d+1}$ and radial projection define a $C_i$-equivariant map $\theta'_d:\Delta^{i-1}/C_i\cdot\Delta^{i-m}\to D(\lambda_{d+1})/(S(\lambda_{d+1})\backslash U)$, where $D(\lambda_{d+1})$ denotes the disk in $\lambda_{d+1}$. \cite[$\S$3]{H1} shows that there is a strong deformation retract of $C_i$-spaces 
$$(S^0*C_m)\wedge S^{\lambda_d}\to D(\lambda_{d+1})/(S^{\lambda_{d+1}}\backslash U).$$
Therefore we get a homotopy class of $C_i$-equivariant maps
$$\theta_d:\Delta^{i-1}/C_i\cdot\Delta^{i-m}\to(S^0*C_m)\wedge S^{\lambda_d}.$$

We also recall some well-known theorems. For $A$ a commutative ring, Hesselholt-Madsen shows that, in \cite[Theorem 7.1]{HM2}, there is an equivalence
\begin{equation}\operatorname{THH}(A[x]/(x^k))\simeq \operatorname{THH}(A)\otimes \operatorname{N^{cy}}(\Pi_k) \tag{a}.
\end{equation}
This equivalence gives rise to $$\operatorname{THH}(A[x]/(x^k), (x))\simeq \bigvee_{i>0} \operatorname{THH}(A)\otimes  \operatorname{N^{cy}}(\Pi_k, i).$$
Here is a corollary of the famous theorem by Dundas-Goodwillie-McCarthy. For a ring $A$ in which $p$ is nilpotent, the cyclotomic trace map \begin{equation}K(A[x]/(x^k), (x))\to \operatorname{TC}(A[x]/(x^k), (x))\tag{b}\end{equation} is an equivalence.

In the present paper, using above theorems, we study the map $$K(A[x]/(x^k), (x))\to K(A[x]/(x^{nk}), (x)),$$ induced by $x\mapsto x^n$, for an $\mathbb{F}_p$-algebra $A$.


\section{The geometric Verschiebung map}
In order to study the map $K(A[x]/(x^k), (x))\to K(A[x]/(x^{nk}), (x))$, we  use two pointed commutative monoids $\Pi_k$ and $\Pi_{nk}$ and their realizations of cyclic bar constructions, and construct a map between corresponding cofibration sequences.

In \cite[7.2]{HM2},  Hesselholt and Madsen defined an isomorphism between the geometric realization $|\Lambda[n][-]|$ of the standard cyclic set and the product topological space $\mathbb{T}\times\Delta^n$ of the circle and the standard n-simplex as follows;\\
In \cite[Theorem 3.4]{J}, Jones constructed an homeomorphism between $|\Lambda[n][-]|$ and $\mathbb{T}\times\Delta^n$ and defined an action of $C_{n+1}$ on $\mathbb{T}\times\Delta^n$ by $$\tau_n\cdot(x;u_0, ..., u_n):=(x-u_0;u_1, ..., u_n, u_0).$$  However, Hesselholt and Madsen consider a different action of $C_{n+1}$ on $\mathbb{T}\times\Delta^n$ given by $$\tau_n*(x;u_0, ..., u_n):=(x-1/(n+1);u_1, ..., u_n, u_0),$$ and defined an $\mathbb{T}\times C_{n+1}$-equivariant homeomorphism $F_n:\mathbb{T}\times\Delta^n\to\mathbb{T}\times\Delta^n$ by $$F_n(x;u_0, ..., u_n):=(x-f_n(u_0, ..., u_n);u_0, ..., u_n)$$ with an affine map $f_n:\Delta^n\to\mathbb{R}$
$$f_n(u_1, ..., u_n, u_0)-f_n(u_0, ..., u_n)=1/(n+1)-u_0,$$ and $$f_n(1, 0, ..., 0)=0.$$ By construction, the restriction $F_n|_{\Delta^n}$ is the identity map.
We identify $|\Lambda[n]|$ with $\mathbb{T}\times\Delta^n$ via this isomorphism. We define a map $\operatorname{e_{i,n}}:\Delta^{i-1}\to\Delta^{in-1},$ which sends the vertex $(0, ..., 0, 1, 0, ..., 0)$, where the (m+1)th coordinates is 1, to the vertex $(0, ..., 0, 1, 0, ..., 0)$, where the (mn+1)th coordinate is 1,  for every $m\in\{0, ..., i-1\}$. In other words, $\operatorname{e_{i,n}}(\xi_i^j)=\xi_{in}^{nj}$ for $0\leq j\leq i-1$, where $\xi_i$, respectively $\xi_{in}$, is the generator of $C_i$, respectively $C_{in}$. 


\begin{lem}\label{benri}The map $\operatorname{e_{i,n}}$ induces the map $\operatorname{g_{i,n}}:\operatorname{N^{cy}}(\Pi_k, i)\to \operatorname{N^{cy}}(\Pi_{nk}, in)$, $a\mapsto b^n$, via the isomorphism, where $\Pi_k$, respectively $\Pi_{nk}$, is generated by $a$, respectively $b$.
\end{lem}

\begin{proof}By \cite[Lemma 2.2.6]{HM1}, the map $\Lambda[i-1][-]\xrightarrow\alpha \operatorname{N^{cy}}(\Pi_k, i)[-]$ representing the $i-1$-simplex $a\wedge...\wedge a$ ($i$ factors) induces a $\mathbb{T}$-equivariant homeomorphism after the geometric realization. We write $\Lambda[in-1]\xrightarrow\beta \operatorname{N^{cy}}(\Pi_{kn}, in)[-]$ for the map representing the $in-1$-simplex $b\wedge...\wedge b$ (in factors). Then we have the following commutative diagram
\[
   \xymatrix{
   \operatorname{N^{cy}}(\Pi_k, i)[-] \ar[r]^-{g'_{i,n}}&\operatorname{N^{cy}}(\Pi_{kn}, in)[-]\\
   \Lambda[i-1]\ar[r]^-{\Psi}\ar[u]^-{\alpha} & \Lambda[in-1] \ar[u]^-{\beta},
}
\]
where the map $\Psi$ of cyclic sets is the one induced by the composition map $d_{in-1}d_{in-2}...d_2d_1$ except $d_{nj}$ for all $j\in\{1, ..., i-1\}$ and $g'_{i,n}$ is the map of cyclic sets that is given by $a\mapsto b^n$ and induces $g_{i,n}$ via the geometric realization by definition. The geometric realization of $\Psi$ with Hesselholt-Madsen's isomorphism mentioned above is given by 
$$\mathbb{T}\times\Delta^{i-1}\to\mathbb{T}\times\Delta^{in-1},$$$$ (t, (u_0, u_1, ..., u_{i-1}))\mapsto(t, (u_0, 0, 0, ...,0, u_1, 0, ..., 0, u_{i-1}, 0, 0, ..., 0)),$$where there are n-1 zeros between $u_{s-1}$ and $u_s$. By definition,  it is the map $\operatorname{id}_{\mathbb{T}}\times\operatorname{e_{i,n}}$. In other words, we have the commutative diagram
\[
   \xymatrix{
  \mathbb{T}\times\Delta^{i-1} \ar[rr]^-{\operatorname{id}_{\mathbb{T}}\times\operatorname{e_{i,n}}} &&\mathbb{T}\times\Delta^{in-1} \\
   |\Lambda[i-1]|\ar[rr]^-{|\Psi|}\ar[u]_-{\cong} && |\Lambda[in-1]| \ar[u]_-{\cong}.
}
\]
By the definition of the cyclic bar construction and the commutativity of our monoids, 
\begin{eqnarray*}
\beta_{i-1}(d^1d^2...d^{in-2}d^{in-1})&=&\beta_{i-1}\circ D(id_{[in-1]})\\
 &=&D_*\circ\beta_{in-1}(id_{[in-1]})\\
 &=&D_*(b\wedge...\wedge b)=b^n\wedge...\wedge b^n,
\end{eqnarray*}
\[
   \xymatrix{
   \Lambda[in-1][in-1] \ar[r]^-{\beta_{in-1}} \ar[d]^-D &\operatorname{N^{cy}}(\Pi_{kn}, in)[in-1] \ar[d]^-{D_*}\\
   \Lambda[in-1][i-1]\ar[r]^-{\beta_{i-1}} & \operatorname{N^{cy}}(\Pi_{kn}, in)[i-1],
}
\]
where $D$ is the image of the map $d_{in-1}d_{in-2}...d_2d_1$ except $d_{nj}$ for all $j\in\{1, ..., i-1\}$ by the contravariant functor $\Lambda[in-1][-]$ and $D_*$ is the image of the map $d_{in-1}d_{in-2}...d_2d_1$ except $d_{nj}$ for all $j\in\{1, ..., i-1\}$ by the contravariant functor $\operatorname{N^{cy}}(\Pi_{kn}, in)$.
\end{proof}

We now study the relation between the map $\operatorname{g}_{i,n}$ and the cofibration sequences above.
More precisely, we have two cofibration sequences for every $i>0$,

\[
   \xymatrix{
 \mathbb{T}_+\wedge_{C_{i/k}}S^{\lambda_{d_i}} \ar[r]^-{\operatorname{pr}} & \mathbb{T}_+\wedge_{C_{i}}S^{\lambda_{d_i}} \ar[r] & \operatorname{N^{cy}}(\Pi_{k}, i),\\
  \mathbb{T}_+\wedge_{C_{in/kn}}S^{\lambda_{d_i}} \ar[r]^-{\operatorname{pr}} & \mathbb{T}_+\wedge_{C_{in}}S^{\lambda_{d_i}} \ar[r] & \operatorname{N^{cy}}(\Pi_{kn}, in),
  }
\]
and are comparing them using $\operatorname{g}_{i,n}$.



\begin{prop}\label{benri2}{\rm (i)}\;For $kd<i<k(d+1)$, the following diagram of $C_i$-spaces commutes up to homotopy

\[
   \xymatrix{
   \Delta^{i-1}/C_i\cdot\Delta^{i-k} \ar[r]^-{\theta_{i,k}} \ar[d]^-{\operatorname{e_{i,n}}_*}&S^{\lambda_d} \ar[d]^-{\operatorname{id}}\\
   \Delta^{in-1}/C_{in}\cdot\Delta^{n(i-k)} \ar[r]^-{\theta_{in,kn}} & S^{\lambda_d}.\\
}
\]

{\rm (ii)}\;For $i=k(d+1)$, the following diagram of $C_i$-spaces commutes up to homotopy

\[
   \xymatrix{
   \Delta^{i-1}/C_i\cdot\Delta^{i-k} \ar[r]^-{\theta_{i,k}} \ar[d]^-{\operatorname{e_{i,n}}_*}&(S^0\ast C_k)\wedge S^{\lambda_d} \ar[d]\\
   \Delta^{in-1}/C_{in}\cdot\Delta^{n(i-k)} \ar[r]^-{\theta_{in,kn}} & (S^0\ast C_{nk})\wedge S^{\lambda_d},\\
}
\]
where the right hand side vertical map is induced by the inclusion $C_k\to C_{kn}$, $\xi_k\mapsto\xi_{nk}^k$ and the identity map on $S^{\lambda_d}$.
\end{prop}

\begin{proof}
We prove (i). The same argument holds for (ii). By the construction of $\theta$ (\cite[$\S$3]{H1}), $\theta_{i,k}(\xi_i^j)=[\xi_i^j, \xi_i^{2j}, ..., \xi_i^d]$, where $\xi_i$ is the generator of $C_i$. Likewise, $\theta_{in,kn}(\xi_{in}^j)=[\xi_{in}^j, \xi_{in}^{2j}, ..., \xi_{in}^d]$, where $\xi_{in}$ is the generator of $C_{in}$. By the definition of $\operatorname{e}$, we have $\operatorname{e}(\xi_{i})=\xi^n_{in}$. In the complex numbers plane $\mathbb{C}$, $\xi_{i}^j=\xi_{in}^{nj}$. 
\end{proof}




By this proposition, we get the following map of cofiber sequences. 

\begin{cor}\label{daiji}There is a homotopy commutative diagram of cofibration sequences
\[
   \xymatrix{
  \mathbb{T}_+\wedge_{C_{i/k}}S^{\lambda_d} \ar[r]^-{\operatorname{pr}} \ar[d]^-{\operatorname{id}} & \mathbb{T}_+\wedge_{C_i}S^{\lambda_d} \ar[r] \ar[d]^-{\operatorname{pr}} & \operatorname{N^{cy}}(\Pi_k, i) \ar[d]^-{g_{i,n}}\\  
 \mathbb{T}_+\wedge_{C_{in/nk}}S^{\lambda_d} \ar[r]^-{\operatorname{pr}} & \mathbb{T}_+\wedge_{C_{in}}S^{\lambda_d} \ar[r]  & \operatorname{N^{cy}}(\Pi_{nk}, ni),\\
}
\]
where $\mathbb{T}_+\wedge_{C_{i/k}}S^{\lambda_d}$ and $\mathbb{T}_+\wedge_{C_{in/kn}}S^{\lambda_d}$ are trivial when $k$ does not divide $i$ and $d=\lfloor (i-1)/k\rfloor$.
\end{cor}


\begin{proof}
Again by \cite[(3.1.1)]{HM1}, the map $\Lambda[j-1][-]\to\operatorname{N^{cy}}(\Pi_m, j)[-]$ representing $y\wedge y\wedge...\wedge y$ ($j$ factors) with the generator $y$ of $\Pi_m$ induces a $\mathbb{T}$-equivariant homeomorphism $$\mathbb{T}_+\wedge_{C_j}(\Delta^{j-1}/C_j\cdot\Delta^{j-m})\to\operatorname{N^{cy}}(\Pi_m, j).$$ We can get two cofibration sequences \[
   \xymatrix{
  \mathbb{T}_+\wedge_{C_{i/k}}S^{\lambda_d} \ar[r]^-{\operatorname{pr}} & \mathbb{T}_+\wedge_{C_i}S^{\lambda_d} \ar[r] & \operatorname{N^{cy}}(\Pi_k, i), \\
   \mathbb{T}_+\wedge_{C_{in/kn}}S^{\lambda_d} \ar[r]^-{\operatorname{pr}} & \mathbb{T}_+\wedge_{C_{in}}S^{\lambda_d} \ar[r] & \operatorname{N^{cy}}(\Pi_{kn}, {in}),
  }
\]
applying $\mathbb{T}_+\wedge_{C_i}(-)$ and $\mathbb{T}_+\wedge_{C_{in}}(-)$ respectively to the diagrams in \ref{benri2}. The inclusion map $C_i\to C_{ni}$, $\xi_i^j\mapsto\xi_{in}^{nj}$, induces the maps $\operatorname{id}$, $\operatorname{pr}$ and $g_{i,n}$ which make the diagram commutative.
\end{proof}

\section{Proof of theorems}
For a ring $A$, the topological Hochschild homology $\operatorname{THH}(A)$ is a cyclotomic spectrum. See, for example, \cite[III.4, III.5]{NS}. For a finite dimensional orthogonal $\mathbb{T}$-representation $\lambda$, we let $S^{\lambda}$ denote the one point compactification. We define $$\operatorname{TR}_{q-\lambda}^n(A):=[S^q\otimes(\mathbb{T}/C_n)_+, S^{\lambda}\otimes\operatorname{THH}(A)]_{\mathbb{T}}$$ to be the abelian group of maps in the $\mathbb{T}$-stable homotopy category (\cite[p.6-7]{H1}). 

Using the diagram in the corollary above, we can get a map of long exact sequences to study commutative rings.
\begin{thm}Let $A$ be a commutative ring and $k$ a positive integer. There is a map of long exact sequences\vspace{-1mm}
\[
   \xymatrix{
 \vdots\ar[d]&\vdots\ar[d]\\
 \prod_{i\geq 1} \operatorname{TR}^{i/k}_{q-\lambda_{\lfloor(i-1)/k\rfloor}}(A) \ar[d]^-{{V_k}_*} \ar[r]^-{\operatorname{id}} &      \prod_{i\geq 1} \operatorname{TR}^{i/kn}_{q-\lambda_{\lfloor(i-1)/kn\rfloor}}(A)     \ar[d]^-{{V_{kn}}_*} \\  
   \prod_{i\geq 1} \operatorname{TR}^{i}_{q-\lambda_{\lfloor(i-1)/k\rfloor}}(A)                 \ar[r]^-{{V_{n}}_*} \ar[d]& \prod_{i\geq 1} \operatorname{TR}^{i}_{q-\lambda_{\lfloor(i-1)/kn\rfloor}}(A) \ar[d]  \\
 \operatorname{TF}_{q+1}(A[x]/(x^k), (x)) \ar[r] \ar[d] &     \operatorname{TF}_{q+1}(A[x]/(x^{nk}), (x))  \ar[d] \\
 \vdots& \vdots\\
}
\]
where $V$ maps are given by Verschiebung maps, $e_i=\lfloor(i-1)/kn\rfloor$, and $\operatorname{TR}^{i/l}_s(A)$ is trivial when $i/l\not\in\mathbb{N}$.
\end{thm}

\begin{proof}
Taking the infinite coproduct of the diagram in the corollary, we get the following map of cofibration sequences
\[\xymatrix{
  \bigvee_{i\geq 0}\mathbb{T}_+\wedge_{C_{i/k}}S^{\lambda_{d_i}} \ar[r]^-{\operatorname{pr}} \ar[d]^-{\operatorname{id}} & \bigvee_{i\geq 0}\mathbb{T}_+\wedge_{C_i}S^{\lambda_{d_i}} \ar[r] \ar[d]^-{\operatorname{pr}} & \operatorname{N^{cy}}(\Pi_k) \ar[d]^-{\operatorname{g_{n}}} \\
  \bigvee_{i\geq 0}\mathbb{T}_+\wedge_{C_{i/kn}}S^{\lambda_{e_i}} \ar[r]^-{\operatorname{pr}} & \bigvee_{i\geq 0}\mathbb{T}_+\wedge_{C_{i}}S^{\lambda_{e_i}} \ar[r]  & \operatorname{N^{cy}}(\Pi_{nk}),\\
}\]
where $\operatorname{g_{n}}$ denotes the map induced by $\Pi_k\to\Pi_{nk}$, $a\mapsto b^n$. Applying the functor $\operatorname{THH}(A)\otimes_{\mathbb{S}}(-)$ to the diagram above and taking fixed points and  the homotopy limits along with Frobenius maps and homotopy groups of spectra, we get the desired diagram by (a).
\end{proof}

We can deduce another map of another long exact sequences by the cofibration sequence.
\begin{thm}\label{THM1}Let $A$ be an algebra in which $p$ is nilpotent. There is a map of long exact sequences\vspace{-1mm}
\[
   \xymatrix{
   \vdots\ar[d]&\vdots\ar[d]&\\
  \lim_{R} \operatorname{TR}^{i/k}_{q-\lambda_{d_{i,k}}}(A) \ar[d]^-{{V_k}_*} \ar[r]^-{\operatorname{id}} &       \lim_{R} \operatorname{TR}^{i/kn}_{q-\lambda_{d_{i, kn}}}(A)            \ar[d]^-{{V_{kn}}_*}  \\  
  \lim_{R} \operatorname{TR}^{i}_{q-\lambda_{d_{i,k}}}(A)       \ar[d]           \ar[r]^-{{V_n}_*} & \lim_{R} \operatorname{TR}^{i}_{q-\lambda_{d_{i, kn}}}(A)   \ar[d] \\
  K_{q+1}(A[x]/(x^k), (x)) \ar[r]^-{v_n} \ar[d] & K_{q+1}(A[x]/(x^{nk}), (x))  \ar[d]  \\
   \vdots     & \vdots      &,\\
}
\]
where $d_{i,k}=\lfloor(i-1)/k\rfloor$, $d_{i, kn}=\lfloor(i-1)/kn\rfloor$, $V$ denotes the map induced by Verschiebung maps and the maps in the limits are restriction maps.
\end{thm}

\begin{proof}
The same argument in the proof of \cite[2.1]{H1} holds. More precisely, we first smash $\operatorname{THH}(A)$ with the cofibration sequences in the proof of the above theorem and use $(a)$. Next we take homotopy limits along with Frobenius maps and homotopy fixed points of restriction maps. Then by $(b)$, we get the desired diagram.
\end{proof}
Taking the colimit of the diagram in the above theorem, we get the following

\begin{cor}Let $A$ be an $\mathbb{F}_p$-algebra and let $\tilde{A}:=\operatorname{colim}_n(A[x]/(x^{nk}), (x))$. Then there is a long exact sequence
\[ \begin{aligned}
{} & \xymatrix{
 { \cdots } \ar[r] &  \operatorname{lim}_R\operatorname{TR}^{i/k}_{q-\lambda_{d_{k}^i}}(A) \ar[r] & \operatorname{colim}_n\operatorname{lim}_R\operatorname{TR}^{i}_{q-\lambda_{d^i_{kn}}}(A) \\ } \\
{} & \xymatrix{
\phantom{\cdots} \ar[r] & K_{q+1}(\tilde{A}) \ar[r] & \cdots. \\ } \\
\end{aligned}
\]

\end{cor}

\begin{proof}The colimit is filtered. Therefore, we get a long exact sequence by taking the colimit of the long exact sequences obtained in the above theorem and $\operatorname{colim}_nK_*(A[x]/(x^{nk}), (x))$ is canonically isomorphic to $K_{*}(\tilde{A})$.
\end{proof}

For any $\mathbb{Z}_{(p)}$-algebra $A$, there is a decomposition by \cite[$\S$2]{H1}
$$\operatorname{TR}^i_{q-\lambda_d}(A)\xrightarrow{\cong}\prod_{j|i'}\operatorname{TR}^u_{q-\lambda_{\lfloor(p^{u-1}j-1)/k\rfloor}}(A;p),$$
where $i=p^{u-1}i'$ with $i'/p\notin\mathbb{N}$ and $d=\lfloor(i-1/k)\rfloor$.
The $i$-th component of the above isomorphism is given by the composition
$$\operatorname{TR}^i_{q-\lambda_d}(A)\to \operatorname{TR}^{i/j}_{q-\lambda_d}\to \operatorname{TR}^{p^{u-1}}_{q-\lambda_{(p^{u-1}j-1)/k}}(A),$$
of the the Frobenius map $F_j$ and the restriction map $R_{i'/j}$. We define $\operatorname{TR}^{u}_{q}(A;p):=\operatorname{TR}^{p^{u-1}}_{q}(A)$.

Let $A$ be a $\mathbb{Z}_{(p)}$-algebra for a prime number $p$, and $i$, $n$ and $q$ natural numbers. Write $i=p^{u-1}i'$ with $i'/p\notin\mathbb{N}$ and $n=p^{v-1}n'$ with $n'/p\notin\mathbb{N}$. Then there is a commutative diagram, see \cite{H1},
\[
   \xymatrix{\operatorname{TR}^i_{q-\lambda}(A) \ar[d] \ar[r]& \prod_{j|i'}\operatorname{TR}^u_{q-\lambda'}(A;p)\ar[d] \\
\operatorname{TR}^{in}_{q-\lambda}(A)\ar[r] & \prod_{j|i'n'}\operatorname{TR}^{u+v-1}_{q-\lambda'}(A;p), \\      
}
\]
where the left vertical map is $n$-th Verschiebung map $V^n$ and the right vertical map acts on the $j$ factor as $n'V^{v-1}$ which lands on $jn'$ factor and the horizontal maps are isomorphisms defined above. We use the stability lemma \cite[Lemma 2.6]{H1}.

\begin{lem}\label{antei}Let $p$ be a prime number and $A$ a $\mathbb{Z}_{(p)}$-algebra and $i=p^{u-1}i'$, $n=p^{v}n'$ and $q$ natural numbers and $j\in I_p$. Then there is a natural number $u'$ such that the following diagram is commutative 
\[
   \xymatrix{\operatorname{lim}_R\operatorname{TR}^u_{q-\lambda_{\lfloor(p^{u-1}j-1)/k\rfloor}}(A;p) \ar^-{\operatorname{pr}}[d] \ar^-{V^{p^v}_*}[r]& \operatorname{lim}_R\operatorname{TR}^u_{q-\lambda_{\lfloor(p^{u-1}j-1)/k\rfloor}}(A;p)\ar[d]^-{\operatorname{pr}} \\
\operatorname{TR}^{u'}_{q-\lambda_{\lfloor(p^{u'-1}j-1)/k\rfloor}}(A;p)\ar[r]^-{V^{p^v}_*} & \operatorname{TR}^{u'+v}_{q-\lambda_{\lfloor(p^{u'-1}j-1)/k\rfloor}}(A;p), \\      
}
\]
and vertical maps are isomorphisms and $q<2\lfloor(p^{u'}j-1)/k\rfloor$.
\end{lem}

\begin{proof}We first note that $\lfloor(p^{u'+v-1}jn'-1)/kn\rfloor=\lfloor(p^{u'-1}j-1)/k\rfloor$. Therefore, the diagram is commutative. Moreover, by Lemma 2.6 in \cite{H1}, the vertical maps are isomorphisms for $q<2\lfloor(p^{u'}j-1)/k\rfloor$.
\end{proof}

Taking the limit on the decomposition, we get the isomorphism $$\operatorname{lim}_R\operatorname{TR}^i_{q-\lambda_d}\xrightarrow{\cong}\prod_{j\in I_p}\lim_R\operatorname{TR}^u_{q-\lambda_{(p^{u-1}j-1)/k}}(A;p),$$ where $I_p$ is the set of natural numbers which are not divided by $p$.

\begin{cor}
Let $A$ be a $\mathbb{Z}_{(p)}$-algebra and $\tilde{A}:=\operatorname{colim}_n(A[x]/(x^{nk}), (x))$. Then we can chose $\tilde{u}$ such that the following is a long exact sequence
\[ \begin{aligned}
{} &   \xymatrix{
\phantom{\cdots}  { \cdots } \ar[r] & \bigoplus_{j\in I_p}\operatorname{lim}_R\operatorname{TR}^{\tilde{u}}_{q-\lambda_{d^{k, p}_{u, j}}}(A;p)\\} \\
{} & \xymatrix{
\phantom{\cdots} \ar[r]&  \bigoplus_{j\in I_p}\operatorname{colim}_v\operatorname{lim}_R\operatorname{TR}^{\tilde{u}+v+s}_{q-\lambda_{d^{k, p}_{u, j}}}(A;p) \ar[r] & \\} \\
{} & \xymatrix{   
\phantom{\cdots} K_{q+1}(\tilde{A}) \ar[r]  &...,\\ }\\
\end{aligned}
\]
where $k=p^sk'$ and $n=p^vn'$ and $d^{k, p}_{u, j}=\lfloor(p^{u-1}j-1)/k\rfloor$.
\end{cor}

\begin{proof}By \ref{antei}, for any $j\in I_p$, it is able to chose large enough $\tilde{u}$ such that the following commutes and vertical maps are isomorphisms
\[
   \xymatrix{\operatorname{lim}_R\operatorname{TR}^u_{q-\lambda_{\lfloor(p^{u-1}j-1)/k\rfloor}}(A;p) \ar^-{\operatorname{pr}}[d] \ar^-{V^{p^s}_*}[r]& \operatorname{lim}_R\operatorname{TR}^{u+s}_{q-\lambda_{\lfloor(p^{u-1}j-1)/k\rfloor}}(A;p)\ar[d]^-{\operatorname{pr}} \\
\operatorname{TR}^{\tilde{u}}_{q-\lambda_{\lfloor(p^{u'-1}j-1)/k\rfloor}}(A;p)\ar[r]^-{V^{p^s}_*} & \operatorname{TR}^{\tilde{u}+s}_{q-\lambda_{\lfloor(p^{u'-1}j-1)/k\rfloor}}(A;p), \\      
}
\]
where $k=p^sk'$. We also have the decomposition for $\mathbb{Z}_{(p)}$-algebra. Therefore, we get the desired long exact sequence.
\end{proof}

For regular $\mathbb{F}_p$-algebras,  in \cite[$\S$5]{H1} Hesselholt gave an explicit translation of topological Hochschild homology and big de Rham-Witt complex $\mathbb{W}_{(-)}\Omega_A^*$. By the translation, we immediately get the following from Theorem \ref{THM1}
\begin{cor}\label{honnyaku}Let $A$ be a regular $\mathbb{F}_p$-algebra. There is a map of long exact sequences
\vspace{-2mm}\[
   \xymatrix{
   \vdots\ar[d]&\vdots\ar[d]&\\
   \bigoplus_{l\geq0}\mathbb{W}_{l+1}\Omega_{A}^{q-2l} \ar[d]^-{{V_k}_*} \ar[r]^-{\operatorname{id}}       &       \bigoplus_{l\geq0}\mathbb{W}_{l+1}\Omega_{A}^{q-2l}      \ar[d]^-{{V_{kn}}_*}  \\  
     \bigoplus_{l\geq0}\mathbb{W}_{k(l+1)}\Omega_{A}^{q-2l}                      \ar[r]^-{{V_n}_*} \ar[d]&           \bigoplus_{l\geq0}\mathbb{W}_{kn(l+1)}\Omega_{A}^{q-2l}   \ar[d]\\
   K_{q+1}(A[x]/(x^k), (x)) \ar[r]^-{v_n} \ar[d] & K_{q+1}(A[x]/(x^{nk}), (x)) \ar[d] \\
 \vdots& \vdots &,\\
}
\]

where the subscript $m(l+1)$ is the truncation set $\{1, 2, ..., m(l+1)\}$ and $V$'s denote the maps induced by Verschiebung's.
\end{cor}

By \cite[Theorem A]{HM1} and our result Corollary \ref{honnyaku}, we have the following commutative diagram of short exact sequences for any $j$
\[
   \xymatrix{
0  \ar[r] &\mathbb{W}_j(k) \ar[d]^-{\operatorname{id}} \ar[r]^-{V_m}&\mathbb{W}_{jm}(k)\ar[r]\ar[d]^-{V_n} &K_{2j-1}(k[x]/(x^m), (x))\ar[d] \ar[r]&0 \\
0\ar[r] &\mathbb{W}_j(k) \ar[r]^-{V_{mn}}&\mathbb{W}_{jmn}(k)\ar[r]&K_{2j-1}(k[x]/(x^{nm}), (x))\ar[r]&0, \\      
}
\]
where $\mathbb{W}$ denotes the big Witt vectors, and the most right vertical map is induced by the power map $x\mapsto x^n$. In the rest of this section, we consider applications of this diagram.

Let $k$ be a perfect field with characteristic $p>0$ and let 
$$W(k)[p^{1/p^{\infty}}]:=\operatorname{colim}_nW(k)[p^{1/p^{n}}],$$
where the structure maps are given by $p^{1/p^n}\mapsto(p^{1/p^{n+1}})^p$. We consider the completion $\mathcal{O}_K:=W(k)[p^{1/p^{\infty}}]^{\wedge}$ with quotient field $K:=\mathcal{O}_K[1/p]$ and residue field $k$. For example, if $k=\mathbb{F}_p$ then $\mathcal{O}_K=\mathbb{Z}_p[p^{1/p^{\infty}}]^{\wedge}$ with quotient field $\mathcal{O}_{K}[1/p]=\mathbb{Q}_p(p^{1/p^{\infty}})^{\wedge}$. Using $W(k)/p=k$ and $W(k)[x]/(x^{p^n}-p)=W(k)[p^{1/p^n}]$, we have 
$$\mathcal{O}_K/p\mathcal{O}_K=\operatorname{colim}_nk[x]/(x^{p^n})$$
with structure maps given by $x\mapsto x^p$. We let $\mathfrak{m}\subset\mathcal{O}_K$ denote the maximal ideal. Taking the colimit of the diagram above, we obtain a corollary.
 
\begin{cor}With the notation above, we have 
\[
   \xymatrix{
   K_{2j-1}(\mathcal{O}_K/p\mathcal{O}_K, \mathfrak{m}/p\mathcal{O}_K)=\operatorname{colim}_n(\mathbb{W}_{jp^n}(k)/V_{p^n}\mathbb{W}_j(k)),\\      
}
\]
where the colimit is indexed by the category of natural numbers under addition. Moreover, the relative $K$-groups in even degrees are zero.
\end{cor}

Let $k$ again be a perfect field with characteristic $p>0$ and let 
$$W(k)[\zeta_{p^{\infty}}]:=\operatorname{colim}_nW(k)[\zeta_{p^n}],$$
where $\zeta_{p^n}$ denotes a primitive $p^n$-th root of unity and we choose these to satisfy $\zeta_{p^n}^p=\zeta_{p^{n-1}}$. We consider the completion $\mathcal{O}_K:=W(k)[\zeta_{p^{\infty}}]^{\wedge}$ with quotient field $K=\mathcal{O}_K[1/p]$ and residue field $k$. For example, if $k=\mathbb{F}_p$ then $\mathcal{O}_K=\mathbb{Z}_p[\zeta_{p^{\infty}}]^{\wedge}$ with quotient field $\mathcal{O}_{K}[1/p]=\mathbb{Q}_p({\zeta_{p^{\infty}}})^{\wedge}$. Let us write $K_0=W(k)[1/p]$ and $K_n=K_0(\zeta_{p^n})$. Since $|K_n:K_0|=p^{n-1}(p-1)$ and $\zeta_{p^n}-1$ is a uniformizer, the map
$$k[x]/(x^{p^{n-1}(p-1)})\to W(k)(\zeta_{p^n})/p$$
given by $x\mapsto\zeta_{p^n}-1$ is an isomorphism. Moreover, with these isomorphisms, the following diagram
 \[
   \xymatrix{
W(k)(\zeta_{p^n})/p &k[x]/(x^{p^{n-1}(p-1)})\ar[l]_-{\cong}\\
W(k)(\zeta_{p^{n-1}})/p \ar[u] &k[x]/(x^{p^{n-2}(p-1)}) \ar[l]_-{\cong} \ar[u], \\      
}
\]
commutes, where the left vertical map is given by $\zeta_{p^{n-1}}-1\mapsto\zeta_{p^n}^p-1$ and the right vertical map is given by $x\mapsto x^p$.
By this construction, we have 
$$\mathcal{O}_K/p\mathcal{O}_K=\operatorname{colim}_nk[x]/(x^{p^{n-1}(p-1)}).$$
We let $\mathfrak{m}\subset\mathcal{O}_K$ denote the maximal ideal. Taking the colimit of the diagram above, we obtain a corollary again.
 
\begin{cor}With the notation above, we have 
\[
   \xymatrix{
   K_{2j-1}(\mathcal{O}_K/p\mathcal{O}_K, \mathfrak{m}/p\mathcal{O}_K)=\operatorname{colim}_n(\mathbb{W}_{jp^{n-1}(p-1)}(k)/V_{p^{n-1}(p-1)}\mathbb{W}_j(k)),\\      
}
\]
where the colimit is indexed by the category of natural numbers under addition. Moreover, the relative $K$-groups in even degrees are zero.
\end{cor}
The $p$-typical decomposition of the right-hand sides in Corollary 4.7 and Corollary 4.8 is explained in \cite[p.4-5]{H1}.

\section{Acknowkedgement}
I am deeply grateful to Lars Hesselholt for suggesting the topic presented here and helpful conversations. I also wish to express my gratitude to Martin Speirs for valuable comments and to the DNRF Niels Bohr Professorship of Lars Hesselholt for the support.

\end{document}